\documentclass[twoside,11pt]{article}

\usepackage{iclr2021_conference,times}

\usepackage[utf8]{inputenc} 
\usepackage[T1]{fontenc}    
\usepackage{hyperref}       
\usepackage{url}            
\usepackage{booktabs}       
\usepackage{amsfonts}       
\usepackage{nicefrac}       
\usepackage{microtype}      
\usepackage{color}

\usepackage{microtype}
\usepackage{graphicx}
\usepackage{cancel}
\usepackage{booktabs} 
\usepackage[noend]{algpseudocode}
\usepackage{subcaption}

\usepackage{amsmath,amsfonts,amsthm,amssymb}

\usepackage{verbatim}
\usepackage[shortlabels,]{enumitem}
\setlist{nolistsep}
\usepackage{algorithm}
\usepackage{algorithmicx}
\usepackage[normalem]{ulem}

\newtheorem{theorem}{Theorem}
\newtheorem{lemma}[theorem]{Lemma}
\newtheorem{proposition}[theorem]{Proposition}
\newtheorem{corollary}[theorem]{Corollary}

\newtheorem{definition}[theorem]{Definition}

\newtheorem{assumption}{Assumption}

\newcommand{\RR}{\mathbb{R}}      

\newcommand{\ifn}{\mathbf{1}} 
\newcommand{\evp}[2]{\mathbb{E}_{#2} \left[#1\right]} 
\newcommand{\abs}[1]{\left| #1 \right|}

\newcommand{\prp}[2]{\text{Pr}_{#2} \left(#1\right)}

\newcommand{\cGar}{\textsc{SQ}}

\newcommand{\cU}{\mathcal{U}}
\newcommand{\cF}{\mathcal{F}}
\newcommand{\cG}{\mathcal{G}}

\newcommand{\lrp}[1]{\left(#1\right)}
\newcommand{\lrb}[1]{\left[#1\right]}
\newcommand{\lrsetb}[1]{\left\{#1\right\}}

\newtheorem{innercustomgeneric}{\customgenericname}
\providecommand{\customgenericname}{}
\newcommand{\newcustomtheorem}[2]{%
  \newenvironment{#1}[1]
  {%
   \renewcommand\customgenericname{#2}%
   \renewcommand\theinnercustomgeneric{##1}%
   \innercustomgeneric
  }
  {\endinnercustomgeneric}
}

\newcustomtheorem{customthm}{Theorem}
\newcustomtheorem{customlemma}{Lemma}
\newcustomtheorem{customdefn}{Definition}

\usepackage{cleveref}

\makeatletter
\def\blfootnote{\xdef\@thefnmark{}\@footnotetext}
\makeatother

\iclrfinalcopy 

\begin{document}

\title{
$p$-value peeking 
and estimating extrema}

\author{%
  Akshay Balsubramani \\
  Stanford University\\
  \texttt{abalsubr@stanford.edu}
}

\maketitle

\begin{abstract}
A pervasive issue in statistical hypothesis testing is that the reported $p$-values are biased downward by data ``peeking" -- the practice of reporting only progressively extreme values of the test statistic as more data samples are collected. 
We develop principled 
mechanisms to estimate such running extrema of test statistics, which directly address the effect of peeking in some general scenarios. 
\end{abstract}


\section{The problem of peeking}


Consider a scientist trying to test a hypothesis on some huge population of samples $X_1, \dots, X_n$. 
The test statistic $f (X_1, \dots, X_n)$ is estimated by drawing a random sample of the data (say $X_1, \dots, X_t$) to compute the conditional expectation $\evp{f (X_1, \dots, X_n) \mid X_{1:t}}{}$. Assuming a null hypothesis with some given $\evp{f (X_1, \dots, X_n)}{}$, a $p$-value $A_t$ is calculated. 
A pragmatic practitioner with ample computing resources is primarily limited by the availability of data, gathering more samples with time. 
While repeatedly testing all data gathered so far, it is common to ``peek" at the reported $p$-values $P_1, \dots, P_t$ until one is low enough to be significant (say at time $\tau$), and report that $p$-value $P_{\tau} = \min_{s \leq \tau} P_s$, resulting in the reported $p$-value having a downward bias. 

Peeking is a form of \textit{$p$-value hacking} that is widespread in empirical science for appealing reasons -- collecting more data after an apparently significant test result can be costly, and of seemingly questionable benefit. 
It has long been argued that the statistician's opinion should not influence the degree of evidence against the null -- ``the rules governing when data collection stops are irrelevant to data interpretation" \citep{ELS63} -- and that collecting more data and hence evidence should always help, not invalidate, previous results. 
However, standard $p$-value analyses ``depend on the intentions of the investigator" \citep{N00} in their choice of stopping rule. 


But it can be proven that for many common tests, repeating the test long enough will lead the scientist to only report a low enough $p$-value -- classical work recognizes that they are ``sampling to reach a foregone conclusion" \citep{A54}. 
The lamentable conclusion is that peeking makes it much more likely to falsely report significance under the null hypothesis. 

This problem has been addressed by existing theory on the subject. 
A line of work by Vovk and coauthors \citep{V93, SSVV11, VW19} develops the idea of correcting the $p$-values uniformly over time using a ``test martingale," and contains further historical references on this idea. 
As viewed within the context of Bayes factors and likelihood ratios, this has also drawn more recent attention for its robustness to stopping \citep{G18, grunwald2019safe}. 
Such work is based on a martingale-based framework for analyzing $p$-values when peeking is performed in such scenarios, described in Section \ref{sec:pval-impl}. 
The corrected $p$-value is valid for all times, not just the time it is computed -- seeing \emph{at any time} a value of $\delta$ allows rejection of the null at significance level $\delta$. 
This holds irrespective of the details of the peeking procedure. 
In a certain sense, this allows us to peer into the future, giving a null model for the future results of peeking. 

We build on this to introduce a family of peeking-robust sequential hypothesis tests in Sec. \ref{sec:multreprmart} and \ref{sec:aytest}. 
The basic vulnerability of many statistical tests to peeking is that they measure \emph{average} phenomena, which are easily distorted by peeking. 
We develop sequential mechanisms for estimating \emph{extremal} functions of a test statistic. 
These use quantitative diagnostics that track the risk of future peeking under the null with past information, and lead to a general random walk decomposition of possible independent interest (e.g., Theorem \ref{thm:revbachelier}). 
Section \ref{sec:disc} discusses them at length in the context of several previous lines of work. 
Most proofs are deferred to the appendix.

\section{Setup: always valid $p$-values}
\label{sec:pval-impl}

Recalling our introductory discussion, a common testing scenario involving a statistic $f$ tests a sample using the conditional mean over the sample: $N_t := \evp{f (X_1, \dots, X_n) \mid X_{1:t}}{}$. 
The stochastic process $N$ is a \emph{martingale} because  $\forall t: \;\evp{(N_t - N_{t-1}) \mid X_{1:(t-1)}}{} = 0$ \citep{D10}. 
Similarly, a \emph{supermartingale} has differences with conditional mean $\leq 0$. 
A more general and formal definition conditions on the canonical filtration $\cF$ (see \Cref{sec:allproofs}). 

A $p$-value is a random variable $P$ produced by a statistical test such that under the null, 
$\prp{P \leq s}{} \leq s \;\;\forall s > 0$. 
We will discuss this in terms of \emph{stochastic dominance} of random variables. 

\begin{definition}
\label{defn:fsd}
A real-valued random variable $X$ (first-order) stochastically dominates another real r.v. $Y$ (written $X \succeq Y$) if either of the following equivalent statements is true \citep{RockafellarRoyset14}: 
$(a)$ For all $c \in \RR$, $\prp{X \geq c}{} \geq \prp{Y \geq c}{}$.
$(b)$ For any nondecreasing function $h$, $\evp{h(X)}{} \geq \evp{h(Y)}{}$.
Similarly, define $X \preceq Y$ if $-X \succeq -Y$. If $X \preceq Y$ and $X \succeq Y$, then $X \stackrel{d}{=} Y$.
\end{definition}

In these terms, a $p$-value $P$ satisfies $P \succeq \mathcal{U}$, with $\mathcal{U}$ a  $\text{Uniform}([0,1])$ random variable.
This can be described as the quantile function of the test's statistic under the null hypothesis. 

The peeker can choose any random time $\tau$ without foreknowledge, to report the value they see as final -- they choose a \emph{stopping time} $\tau$ (see \Cref{sec:allproofs} for formal definitions) instead of pre-specifying a fixed time $t$. 
So a peeking-robust $p$-value $H_t$ requires that for all stopping times $\tau$, $H_{\tau} \succeq \mathcal{U}$. 
As $\tau$ could be any fixed time, this condition is more strict than the condition on $P$ for a fixed $t$. 
$H$ is an inflated process that compensates for the downward bias of peeking. 

How is the stochastic process $H$ defined? 
There is one common recipe: define $H_t = \frac{1}{M_t}$, using a nonnegative discrete-time (super)martingale $M_t$ 
with $M_0 = 1$. 
This guarantees $H$ is a robust $p$-value process, i.e. $H_{\tau} \succeq \mathcal{U}$ for stopping times $\tau$. 
(The reason why is briefly stated here: the expectation $\evp{M_{\tau}}{}$ is controlled at any stopping time $\tau$ by the supermartingale optional stopping theorem (Theorem \ref{thm:optstoppingsupermart}), so $\evp{ M_{\tau} }{} = \evp{ M_{0} }{} = 1$. 
Therefore, using Markov's inequality on $M_{\tau}$, we have $\prp{ H_{\tau} \leq s}{} = \prp{ M_{\tau} \geq \frac{1}{s}}{} \leq s$. )

Such a ``test [super]martingale" $M_t$ turns out to be ubiquitous in studying sequential inference procedures \citep{SSVV11, VW19}, and is effectively necessary for such inference \citep{RRLK20}. 
\Cref{sec:ztest}
Therefore, our analysis focuses on a nonnegative discrete-time supermartingale $M_t$ 
with $M_0 = 1$. 
We also use the cumulative maximum $S_{t} := \max_{s \leq t} M_{s}$ and the lookahead maximum $S_{\geq t} := \max_{s \geq t} M_s$.

\section{Warm-up: has the ultimate maximum been attained?}
\label{sec:multreprmart}

In the peeking scenario, it suffices to consider times until $\tau_F := \max \lrsetb{ s \geq 0 : M_s = S_s}$, the time of the final attained maximum, because no peeker can report a greater value than they see at this time. 
However, $\tau_F$ is not a stopping time because it involves occurrences in the future, so traditional martingale methods do not study it. 

Studying $\tau_F$ is a useful introduction to the main results of this paper. 
We describe $\tau_F$ by establishing a ``multiplicative representation" of a nonnegative discrete-time (super)martingale $M_t$ 
(with $M_0 = 1$) in terms of its maxima. 

\begin{theorem}[Bounding future extrema with the present]
\label{thm:supermartdecomp}
Define the supermartingale $Z_t := \prp{ \tau_F \geq t \mid \cF_t}{}$.
Then with $\mathcal{U}$ a standard $\text{Uniform}([0,1])$ random variable: 
\begin{enumerate}[(a)]
    \item 
    \label{item:doobineqpastfuture}
    $S_{\geq t} \preceq \frac{M_t}{\mathcal{U}} $. Therefore, $S_{\infty} \preceq \frac{1}{\mathcal{U}}$, and $\forall t$ such that $M_t > 0$, $S_{\infty} \preceq S_t \max \lrp{ 1, \frac{M_t}{S_t} \lrp{ \frac{1}{\mathcal{U}}} } $.
    \item
    $Z_t \leq \frac{M_t}{S_t}$, with equality if $M$ is a martingale.
    \item
    Define $ Q_t := \sum_{i=1}^{t} \lrp{ \frac{ M_{i} - M_{i-1} }{S_{i}} } $ 
    and 
    $ L_t := \sum_{q=1}^{t} M_{q-1} \lrp{ \frac{1}{S_{q-1}} - \frac{1}{S_{q}} }$. 
    Then the decomposition $Z_t \leq 1 + Q_t - L_t$ holds, with equality for martingale $M$. 
    Furthermore: 
    \begin{itemize}
        \item
        $Q$ is a (super)martingale if $M$ is. 
        \item
        $L$ is a nondecreasing process which only changes when $M$ hits a new maximum. 
    \end{itemize} 
\end{enumerate}
\end{theorem}

$Z_t$ is called the \emph{Az\'{e}ma supermartingale} of $M$ \citep{Azema73}. 
Note that $M_{t-1} \leq S_{t-1}$, so that 
\begin{align}
\label{eq:logmaxupp}
L_t \leq \sum_{q=1}^{t} S_{q-1} \lrp{ \frac{1}{S_{q-1}} - \frac{1}{S_{q}} } = \sum_{q=1}^{t} \lrp{ 1 - \frac{ S_{q-1} }{S_{q}} } \leq \sum_{q=1}^{t} \log \lrp{ \frac{ S_{q} }{S_{q-1}} } = \log S_{t}
\end{align}
where we use the inequality $1 - \frac{1}{x} \leq \log x$ for positive $x$. This can be quite tight ($L_t \approx \log S_{t}$) when the steps are small relative to $S_{t-1}$, so that $M_{t-1}$ is not much lower than $S_{t-1}$ at the times $L_t$ changes. 
This decomposition is intimately connected with $\log S_t$, as we will see that the martingale $Q_t$ is effectively equal to $\evp{ \log S_{\infty} \mid \cF_t}{} - 1$ (\Cref{thm:logmaxprocess}). 

Notably, $\frac{M_{t}}{S_{t}}$ can be calculated pathwise, so a natural question is if it can be used as a peeking-robust statistic, i.e. if we can reason about its peeked version 
$$ R_t := \min_{s \leq t} \frac{M_{s}}{S_{s}} \geq \min_{s \leq t} Z_{s} $$
which is a nonincreasing process. 
The following result 
shows that $R_t$ can be considered a valid $p$-value at any time horizon.

\begin{theorem}[An alternative $p$-value]
\label{thm:normpval}
With $\mathcal{U}$ denoting a standard $\text{Uniform}([0,1])$ random variable,
\begin{enumerate}[(a)]
    \item
    For any stopping time $\tau \leq \tau_{F}$, $R_{\tau} \succeq \mathcal{U}$. 
    \item
    Define $\rho_{F} := \max \lrsetb{ t \leq \tau_{F} : Z_{t} = \min_{u \leq \tau_{F} } Z_{u} }$. Then $R_t \geq \prp{ \rho_F > t \mid \cF_t}{}$.
\end{enumerate}
\end{theorem}



\section{Estimating extrema of martingales}
\label{sec:aytest}

For fixed sample sizes, any statistic $T$ with null distribution $\mu$ can be computed from its $p$-value by applying the statistic's inverse complementary CDF $\bar{\mu}^{-1}$ to the $p$-value $P$. 
In this way, we can think of any distribution $\mu$ in terms of a nondecreasing function $g (x) := \bar{\mu}^{-1} (1/x)$ for $x \geq 1$, so that $g \lrp{P}$ corresponds to the statistic $T$. 
In this prototypical case, $T \preceq \mu$. 
Similarly, given a martingale $M$ associated with a robust $p$-value process $H$, the equivalent statistic $g(M) = g \lrp{\frac{1}{H}}$ is dominated by $\mu$. 

Assume $M$ is a martingale and suppose we test a statistic $g^{\mu} (M_t)$ with a process $A_t$. 
The obvious choice $A_t = g^{\mu} (M_t)$ is prone to peeking. 
We instead inoculate $A$ against future peeking by maximizing over the entire trajectory of $A$, and using that as a test statistic. 
We directly estimate the extreme value $\max_{t} g (M_{t}) = g ( S_{\infty} )$ 
-- a quantity robust to peeking -- 
with the process (martingale) $\evp{g ( S_{\infty} ) \mid \cF_t}{} $.\footnote{If the peeker can be assumed to have a limited waiting period of $T$ samples, $S_{\infty}$ can be replaced by $S_{T}$ in this analysis.} 

This quantity has a natural motivation, but it depends on the future through $S_{\infty}$, 
and confounds attempts at estimation with fixed-sample techniques.
Nevertheless, we show how to efficiently compute this as a stochastic process (\Cref{thm:AYpotentialdecomp}), and prove that its null distribution is $\preceq \mu$, under a ``good" stopping rule (\Cref{thm:ayprocessproperties}). 
This characterization leads to results which are more generally novel (Section \ref{sec:univay}).

We also study the interplay between the statistic $\evp{g ( S_{\infty} ) \mid \cF_t}{}$ and its own ``peeked" cumulative maximum $\max_{t} \evp{g ( S_{\infty} ) \mid \cF_t}{}$, characterizing it in terms of $\mu$ (\Cref{thm:ayprocessproperties}, \Cref{thm:givenmuoptisay}) and showing that $\evp{g ( S_{\infty} ) \mid \cF_t}{}$.


\subsection{Estimating the running extremum}
\label{sec:runningmaxtest}

We can use the distributional characterization of Theorem \ref{thm:supermartdecomp} to provide insight into the statistic $\evp{g ( S_{\infty} ) \mid \cF_t}{} $ and ways to compute it. 

\begin{theorem}
\label{thm:AYpotentialdecomp}
For any nondecreasing function $g$, denote $\cG (s) := \int_0^{1} g \left( \frac{s}{u} \right) du$ and its derivative $\cG' (s) := \frac{d \cG (s)}{ds} = \int_{s}^{\infty} \frac{dx}{x^2} \lrp{ g (x) - g (s) }$. 
Then $\cG$ is continuous, concave, and nondecreasing. 
Also:
\begin{align*}
\evp{g ( S_{\infty} ) \mid \cF_t}{} 
\leq Y_t 
&\stackrel{\textbf{(a)}}{:=} \left( 1 - \frac{M_t}{S_t} \right) g (S_t) + M_t \int_{S_t}^{\infty} \frac{g (x)}{x^2} dx \\
&\stackrel{\textbf{(b)}}{=} \left( 1 - \frac{M_t}{S_t} \right) g (S_t) + \frac{M_t}{S_t} \cG (S_t) \\
&\stackrel{\textbf{(c)}}{=} \cG (S_t) - (S_t - M_t) \cG' (S_t) \\
&\stackrel{\textbf{(d)}}{=} g (S_t) + M_t \cG' (S_t)
\end{align*}
with equality when $M$ is a martingale. Furthermore, $g \geq 0 \implies Y \geq 0$.  

\end{theorem}

Theorem \ref{thm:AYpotentialdecomp}$(a)$ shows exactly which choices of $g$ are appropriate, as $\evp{g ( S_{\infty} ) }{} $ can only be bounded if $\frac{g (x)}{x^2}$ is integrable away from zero. 
This paper assumes this hereafter:

\begin{assumption} 
\label{ass:tailintegrab}
$\frac{g (x)}{x^2}$ has a finite integral on any closed interval away from zero.
\end{assumption}

Theorem \ref{thm:AYpotentialdecomp} characterizes the test statistic $Y$, the \emph{Az\'{e}ma-Yor (AY) process} of $M$ with respect to $g$ \citep{AY79}. Thm. \ref{thm:AYpotentialdecomp}(b) can be interpreted as an expectation over two outcomes, using Theorem \ref{thm:supermartdecomp}(b). With probability $1 - \frac{M_t}{S_t}$, the cumulative maximum is not exceeded in the future ($\tau_F \leq t$), so $g (S_{\infty}) = g (S_t)$. Alternatively with probability $\frac{M_t}{S_t}$, the cumulative maximum is exceeded in the future ($\tau_F > t$), and the conditional expectation of $g (S_{\infty})$ in this case is $\cG (S_t)$, using Theorem \ref{thm:supermartdecomp} to get a precise idea of the lookahead maximum from the present.

The AY process $Y$, constructed by Theorem \ref{thm:AYpotentialdecomp} using any $g$, has some remarkable properties that further motivate its use.

\begin{lemma}[Properties of AY processes]
\label{lem:aydifferences}
Define the Bregman divergence $D_F (a, b) := F(a) - F(b) - (a-b) F'(b) \geq 0$ for any convex function $F$. Any AY process $Y$ defined as in Theorem \ref{thm:AYpotentialdecomp} is a supermartingale. The following relations hold pathwise for all $t$: 
\begin{enumerate}[a)]
    \item
    $\displaystyle Y_t \geq g (S_t)$ 
    \item 
    $Y_{t} - Y_{t-1} 
    = (M_{t} - M_{t-1}) \cG' (S_{t-1}) - D_{- \cG} (S_{t}, S_{t-1}) $
    \item
    $\displaystyle \max_{s \leq t} Y_s = \cG (S_t) \geq Y_{t} \geq \cG (M_t)$
    \item
    For any stochastic process $A$, if $A_u \geq \cG(M_u)$ for all $u$, then $\displaystyle \max_{s \leq t} A_s \geq \max_{s \leq t} Y_s$. 
\end{enumerate}
\end{lemma}

\subsection{Consequences and examples}

\begin{theorem}
\label{thm:logmaxprocess}
Define $ Q_t := \sum_{i=1}^{t} \lrp{ \frac{ M_{i} - M_{i-1} }{S_{i}} } $ as in \Cref{thm:supermartdecomp}(c). 
Then $Q_t \leq \evp{\log (S_\infty) \mid \cF_t}{} - 1 $. 
Here the inequality is as tight as that in \eqref{eq:logmaxupp}. 
\end{theorem}
\begin{proof}
First, note that $1 + Q_t = Z_t + L_t$ from \Cref{thm:supermartdecomp}(c). 
Define $ Q_t := \sum_{i=1}^{t} \lrp{ \frac{ M_{i} - M_{i-1} }{S_{i}} } $ from 
Using \Cref{thm:AYpotentialdecomp} with $g(x) = \log (x)$, 
we have $\cG (x) = \log (x) + 1$, so
\begin{align*}
\evp{\log (S_\infty) \mid \cF_t}{} 
&= \log (S_t) + 1 - (S_t - M_t) \frac{1}{S_t}
= \log (S_t) + \frac{M_t}{S_t}
\geq L_t + Z_t
= 1 + Q_t
\end{align*}
\end{proof}

\Cref{thm:AYpotentialdecomp}(d) implies a simple formula for the mean of the ultimate maximum  $\evp{ g( S_\infty ) }{}$. 
\begin{corollary}
With $g$ and $\cG$ defined as in \Cref{thm:AYpotentialdecomp}, $\evp{ g( S_\infty ) }{} = g( 1 ) + \cG' (1)$. 
\end{corollary}

\subsection{Bounding the null distribution}
\label{sec:designnull}

Next, we characterize the null distribution of the test statistic process $Y$.

Our stated motivation for $g$ in Sec. \ref{sec:runningmaxtest} involves a distribution $\mu$, which plays the role of the null in the fixed-sample case. We proceed to specify a stopping time $\tau^{\mu}$ such that the stopped test statistic satisfies the same null guarantee as the fixed-sample one: $Y_{\tau^{\mu}} \preceq \mu$. Our development depends on some properties of $\mu$.

\begin{definition}
\label{defn:supq}
A real-valued distribution $\mu$ has a \textbf{complementary CDF} $\bar{\mu} (x) := \prp{X \geq x}{X \sim \mu}$, a \textbf{tail quantile} function $\bar{\mu}^{-1} (\xi) := \min \lrsetb{ x: \bar{\mu} (x) < \xi }$, and \textbf{barycenter} function $\psi_{\mu} (x) = \evp{X \mid X \geq x }{\mu}$.
Its \textbf{superquantile} function is $\cGar^{\mu} (\xi) := \psi_{\mu} \lrp{ \bar{\mu}^{-1} (\xi) } = \frac{1}{\xi} \int_{0}^{\xi} \bar{\mu}^{-1} (\lambda) d \lambda $, and its \textbf{Hardy-Littlewood transform} is the distribution $\mu^{\textsc{HL}} := \cGar^{\mu} (\mathcal{U})$ for a $[0,1]$-uniform random variable $\mathcal{U}$ \citep{CEO12, RockafellarRoyset14}. 
$\mu$ is associated with a nondecreasing function $ g^{\mu} (x) = \bar{\mu}^{-1} (1/x) $ with corresponding \textbf{future loss potential} $\cG^{\mu} (x) := \int_0^{1} g^{\mu} \left( \frac{x}{u} \right) du = \cGar^{\mu} (1/x)$.
\end{definition}

(Hereafter, superscripts of $\mu$ will be omitted when clear from context.) The characterization provided by Theorem \ref{thm:AYpotentialdecomp} precisely characterizes the mediating function $g$'s effect on the distribution of the given null process $Y$, fully specifying its distribution. 

\begin{theorem}
\label{thm:ayprocessproperties}
Fix a $\mu$ and define $\displaystyle \tau^{\mu} := \min \lrsetb{ t : g^{\mu} (S_t) \geq Y_{t} } $. Then $\displaystyle \max_{s \leq \tau^{\mu} } Y_s \preceq \mu^{\textsc{HL}}$, and $Y_{\tau^{\mu}} \preceq \mu$. 
\end{theorem}



\begin{theorem}
\label{thm:givenmuoptisay}[see also \cite{GM88}]
For any distribution $\mu$, nonnegative martingale $A$, and stopping time $\tau$, if $A_{\tau} \preceq \mu$, then $\displaystyle \max_{s \leq \tau} A_{s} \preceq \mu^{\textsc{HL}} $. 
\end{theorem}

\subsection{Universality}
\label{sec:univay}

Having derived the AY process for any nonnegative supermartingale $M$, we have introduced a number of perspectives on its favorable properties and usefulness as a test statistic (Thm. \ref{thm:AYpotentialdecomp}, Lemma \ref{lem:aydifferences}). This section casts those earlier developments more powerfully, with a converse result: any stochastic process can be viewed as an AY-like process. We know this to be only a loose solution because $Y$ is a strict supermartingale even when $M$ is a martingale (by Lemma \ref{lem:aydifferences}). Instead, a recentered version of this process is appropriate, satisfying two important difference equations pathwise. 

\begin{lemma}
\label{lem:bachelier}
Given any process $M$ and continuous concave nondecreasing nonnegative $\cG$, there is an a.s. unique process $B$ with $B_0 = \cG (1)$ such for all $t$, 
\begin{equation}`
\label{eq:bacheq}
B_{t} - B_{t-1} = (M_{t} - M_{t-1}) \cG' (S_{t-1})
\qquad \text{and} \qquad
\max_{s \leq t} B_s - B_t = (S_t - M_t) \cG' (S_t)
\end{equation}
Due to \eqref{eq:bacheq}, if $M$ is a nonnegative (super)martingale respectively, so is $B$. 
For $t \geq 0$, $B$ is defined by 
\begin{align}
\label{eq:mmdecomp-rev}
B_t := \cG (S_t) - (S_t - M_t) \cG' (S_t) + \sum_{s=1}^{t} D_{- \cG} (S_{s}, S_{s-1})
\end{align}
\end{lemma}
Lemma \ref{lem:bachelier} says that $B$, a bias-corrected version of $Y$ (w.r.t. $M$), is a ``damped" version of $M$ with variation modulated by the positive nonincreasing function $\cG' (S_{t-1})$. This result couples the entire evolutions of $M$ and $B$, so after fixing initial conditions we can derive a unique decomposition of any process $B$ in terms of a martingale $M$ and its cumulative maximum $S$. 

\begin{theorem}[Martingale-max (MM) Decomposition]
\label{thm:revbachelier}
Fix any continuous, concave, strictly increasing, nonnegative $\cG$. Any process $B$ with $B_0 = \cG (1)$ can be uniquely (a.s.) decomposed in terms of a ``variation process" $M$ and its running maximum $S$, such that $M_0 = S_0 = 1$ and \eqref{eq:bacheq} holds. 
The processes $M_t$ and $S_t$ are defined for any $t \geq 1$ inductively by 
\begin{equation}
\label{eq:revbachelier}
M_t = 1 + \sum_{s=1}^{t} \frac{B_s - B_{s-1}}{\cG' (S_{s-1})} \qquad , \qquad S_t = \max_{s \leq t} M_s
\end{equation}
If $B$ is a (super)martingale respectively, so is $M$. 
\end{theorem}

This depends on an attenuation function $\cG$, decomposing the input $B$ into a variation process $M$ and its cumulative maximum $S$, which (as a nondecreasing process) functions as an ``intrinsic time" quantity. Thm. \ref{thm:revbachelier} vastly expands the scope of these analytical tools for AY processes to be applicable to stochastic processes more generally, readily allowing manipulation of cumulative maxima.

\subsection{Max-plus decompositions}




We can also cast the scenario of Section \ref{sec:aytest} in terms of the quantity $\cG (M_t)$. This is a supermartingale if $M$ is \citep{D10}, and many supermartingales can be written in such a form. By Theorem \ref{thm:AYpotentialdecomp}, $\cG (M_t) = \evp{ \frac{M_t}{u} }{u \in \cU} \geq \evp{ g (S_{\geq t}) \mid \cF_t}{} = \evp{ \max_{s \geq t} g (M_s) \mid \cF_t}{}$, where the inequality is by the stochastic dominance relation in Theorem \ref{thm:supermartdecomp}. In our scenario, this can be viewed without further restrictions as a unique decomposition of $\cG (M_t)$, following the continuous-time development (\cite{EM08}, Prop. 5.8).

\begin{theorem}[Max-plus (MP) Decomposition]
\label{thm:maxplus}
Fix any continuous, concave, strictly increasing, nonnegative $\cG$. For any nonnegative martingale $M_t$ with $M_0 = 1$, there is an a.s. unique process $L_t$ such that $\cG (M_t) = \evp{ \max_{s \geq t} L_s \mid \cF_t}{}$, with equality for martingale $M$. This can be written as $L_t := g (M_t)$ for the nondecreasing function $g(x) := \cG (x) - x \cG' (x)$. Also, there is an a.s. unique supermartingale $Y$ with $Y_t \geq \cG (M_t)$ for all $t$ pathwise.
\end{theorem}

\section{Discussion}
\label{sec:disc}

\subsection{Sequential testing}
\label{sec:relwork}

Treating the sample size as a random stopping time is central to the area of sequential testing. 
Much work in this area has focused around the likelihood-ratio martingale of a distribution $f$ for data under a null distribution $g$: $M_t := \displaystyle \prod_{i=1}^{t} \frac{f (X_i)}{g (X_i)}$.
A prototypical example is the Sequential Probability Ratio Test (SPRT, from \cite{wald48}), which is known to stop optimally soon given particular type I and type II error constraints. %
The likelihood-ratio martingale has been explored for stopping in other contexts as well \citep{robbinsNonparam, RS70, BBW97}, including for composite hypotheses \citep{wasserman2020universal, grunwald2019safe}. 
These all deal with specific situations in which the martingale formulation allows for tests with anytime guarantees.

Frequentist or nonparametric perspectives on sequential testing typically contend with LIL behavior. 
For example, the work of \cite{BR16} presents sequential nonparametric two-sample tests in a framework related to ours. 
Such work requires changing the algorithm itself to be a sequential test in an appropriate setting, with a specified level of $\alpha$. 
The setting of $p$-values is in some sense dual to this, as explored in recent work \citep{HRMS18, shin2020nonparametric}. 

Sequential testing involves specifying a type I error a priori (and sometimes also type II, e.g. for the SPRT), 
while what we are reporting is a minimum significance level at which the data show a deviation from the null. 
This is exactly analogous to the relationship between Neyman-Pearson hypothesis testing and Fisher-style significance testing -- the method of this paper can be considered a robust Fisher-style significance test under martingale nulls, just as sequential testing builds on the Neyman-Pearson framework. 
Similarly, we do not analyze any alternative hypothesis, which would affect the power of the test 
(though the choice of test statistic governs the power).

\subsection{Technical tools}

The particulars of computing $H$-values are direct algorithmic realizations of the proof of \cite{B14}, which also shows that these $H$-values are as tight as possible within a constant factor on the probability. The broader martingale mixture argument has been studied in detail in an inverted form, as a uniform envelope on the base martingale $M$ \citep{R52, RS70}. 

In testing maxima, we are guided by the framework fundamentally linking the $\cGar (\cdot)$ function and the maxima of stochastic processes. 
$\cGar$ has been used in much the same time-uniform context (\cite{BD63}, Thm. 3a), and seminal continuous-time contributions showed that this can control the maximum of a continuous martingale in general settings \citep{DG78, AY79}. 
Related work also includes the continuous (super)martingale ``multiplicative representations" of \cite{NY06}, whose techniques we repurpose. 
The modern usage crucially involves a variational characterization of $\cGar$ \citep{RockUCvar00} that would be an interesting avenue to future methods \citep{RockafellarRoyset14}.

Many stopping-time issues in this paper have been studied for Brownian motion, and some for martingales in continuous time under regularity conditions. 
Stopping Brownian motion to induce a given stopped distribution has been well studied in probability, as the Skorokhod embedding problem \citep{obloj2004skorokhod}. 
AY processes were originally proposed as a continuous solution of the Skorokhod problem \citep{AY79}, analogous to our discrete-time results on the null distribution of our AY test statistic, for which we adapted techniques from previous work \citep{GM88, CEO12}. 
The difference equation of Lemma \ref{lem:bachelier} has been studied in the context of future maxima since \cite{Bach1906}. 
To our knowledge the \textsc{MM} decomposition is novel, though in continuous time the AY process can be inverted directly \citep{EM08}.

\subsection{Future work}

The importance of peeking has long been recognized in the practice of statistical testing \citep{R52, AMR69, N00, W07p, SNS11}, mostly in a negative light. 
The statistician typically does not know their sampling plan, which is necessary for standard hypothesis tests. 
The stopping rule is subject to many sources of variation: for example, it could be unethical to continue sampling when a significant effect is detected in a clinical trial \citep{I08}, or the experimenter could run out of resources to gather more data. 
Solutions to this problem are often semi-heuristic and generally involve ``spending a budget of $\alpha$," the willingness to wrongly reject the null, over time. 
Such methods are widely used \citep{Peto77, P77, SAL14} but are not uniformly robust to sampling strategies, and their execution suffers from many application-specific complexities arising from assumptions about the possible stopping times employed by the peeker \citep{pocock05}.  

We hope to have presented general and useful theory to address this state of affairs. 
A main open problem of interest here is applying these results to design and deploy new hypothesis tests. 






\bibliographystyle{plainnat}
\bibliography{main}

\appendix


\newpage

\section{Proofs of results}
\label{sec:allproofs}

\subsection{Preliminaries}
\label{sec:optstopdefn}

In our setting, a \emph{stopping time} is an adapted real function of the past (sub-)$\sigma$-algebra $\cF_t$ (see the works \citep{D10, K06} for more theoretical background). The central result about stopping times, which is the basis of this paper's development, is the \emph{optional stopping theorem}.

\begin{customthm}{0}[Optional Stopping for Nonnegative Supermartingales (\citep{D10}, Theorem 5.7.6)]
\label{thm:optstoppingsupermart}
Let $M$ be a nonnegative supermartingale. Then if $\tau > s$ is a (possibly infinite)
stopping time, $\evp{M_\tau \mid \cF_s}{} \leq M_s$, with equality when $M$ is a martingale.
\end{customthm}

This is typically useful for bounding probabilities pathwise, after applying Markov's inequality on a particular choice of the stopped process $M_\tau$.

\begin{customlemma}{0}[\cite{ville39}]
\label{lem:ville}
If $M$ is a nonnegative supermartingale, for any $c > 0$,
$\prp{ \max_{t \geq s} M_t \geq \frac{M_s}{c} \mid \cF_s }{} \leq c$, with equality for martingale $M$ with $\lim_{t \to \infty} M_t \to 0$ a.s.
\end{customlemma}
\begin{proof}[Proof of Lemma \ref{lem:ville}]
Consider the stopped $M_\tau$ for $\tau := \min  \lrsetb{ t \geq s : M_t \geq \frac{M_s}{c} }$. 
By Theorem \ref{thm:optstoppingsupermart}: 
\begin{align*}
M_{s}
&\geq \evp{M_{\tau}}{}
= \evp{M_{\tau} \mid \tau < \infty, \cF_s }{} \prp{\tau < \infty \mid \cF_s}{} + \evp{M_{\tau} \mid \tau = \infty, \cF_s }{} \prp{\tau = \infty \mid \cF_s}{} \\
&\geq \frac{M_s }{c} \prp{\tau < \infty \mid \cF_s}{}
\end{align*}
proving that $\prp{\tau < \infty \mid \cF_s}{} \leq c$, 
which is the result. The equality case uses the same proof, replacing the inequalities by equalities.
\end{proof}




\subsection{Deferred Proofs}

Here are full proofs of all the 
results we introduce in this paper.

\begin{proof}[Proof of Theorem \ref{thm:supermartdecomp}]
\begin{enumerate}[(a)]
    \item
    By Lemma \ref{lem:ville}, $\forall s \in (0,1)$, $\prp{ \frac{M_t}{S_{\geq t}} \leq s }{} = \prp{ S_{\geq t} \geq \frac{M_t}{s} }{} \leq s = \prp{ \mathcal{U} \leq s }{}$, proving that $S_{\geq t} \preceq \frac{M_t}{\mathcal{U}} $. Taking the nondecreasing function $x \mapsto \max (S_t, x)$ of both sides gives the result.
    \item 
    By Lemma \ref{lem:ville} on $\lrsetb{M_s}_{s=t}^{\infty}$, $ Z_t = \prp{\tau_F \geq t \mid \cF_t}{} = \prp{\max_{s \geq t} M_s \geq S_t \mid \cF_t}{} \leq \frac{M_t}{S_t}$.
    \item
    By part $(b)$, we have $Z_t - Z_0 \leq \frac{M_t}{S_t} - 1$.
    We can write 
    \begin{align*}
    \frac{M_t}{S_t} - 1
    &= \sum_{q=1}^{t} \lrp{ \frac{M_q}{S_q} - \frac{M_{q-1}}{S_{q-1}} }
    = \sum_{q=1}^{t} \lrp{ \frac{M_q}{S_q} - \frac{M_{q-1}}{S_q} + \frac{M_{q-1}}{S_q} - \frac{M_{q-1}}{S_{q-1}} } \\
    &= \sum_{i=1}^{t} \lrp{ \frac{ M_{i} - M_{i-1} }{S_{i}} } - \sum_{q=1}^{t} M_{q-1} \lrp{ \frac{1}{S_{q-1}} - \frac{1}{S_{q}} }
    \end{align*}
\end{enumerate}
\end{proof}

Theorem \ref{thm:supermartdecomp} adapts continuous-time results from \cite{NY05, NY06}.

\begin{proof}[Proof of Theorem \ref{thm:normpval}]
\begin{enumerate}[(a)]
    \item
    Define $\tau (u) := \min \lrsetb{t: Z_t \leq u}$ for $u \in (0,1)$. Then $\forall \; \tau \leq \tau_{F}$, 
    \begin{align*}
    \prp{R_{\tau} \leq u}{}
    &= \prp{\tau (u) \leq \tau}{} \\
    &\leq \prp{\tau (u) \leq \tau_{F}}{} = \evp{ \prp{\tau_{F} \geq \tau (u) \mid \cF_t}{} }{} = \evp{ Z_{\tau (u)} }{}
    \leq u
    \end{align*}
    where the last equality is by definition of $Z$, and the last inequality is by definition of $\tau(u)$.
    \item
    For any $t$, $\ifn \lrp{\rho_F > t} = \ifn \lrp{\exists t < \tau \leq \tau_{F} : Z_{\tau} \leq R_{\tau_{F}} } \leq \ifn \lrp{\exists t < \tau \leq \tau_{F} : Z_{\tau} \leq R_t}$. Taking $\evp{ \cdot \mid \cF_t}{}$ on both sides and defining 
    $\tau_I := \min \lrsetb{\tau > t: Z_{\tau} \leq R_{t}}$, 
    \begin{align*}
    \prp{ \rho_F > t \mid \cF_t}{}
    &\leq \prp{ \lrb{\exists t < \tau \leq \tau_{F} : Z_{\tau} \leq R_t } \mid \cF_t}{} = \prp{ \tau_{I} \leq \tau_{F} \mid \cF_t}{}
    \stackrel{(a)}{=} Z_{\tau_{I}} 
    \stackrel{(b)}{\leq} R_t
    \end{align*}
    where $(a)$ and $(b)$ are respectively by definition of $Z_t$ and $\tau_{I}$.
\end{enumerate}
\end{proof}

\begin{proof}[Proof of Theorem \ref{thm:AYpotentialdecomp}]
We condition on whether $\tau_F \leq t$. Using Theorem \ref{thm:supermartdecomp} (i.e., for a uniform random variable $\mathcal{U}$, $S_{\geq t} \preceq \frac{M_t}{\mathcal{U}}$) and the monotonicity of $g$, 
\begin{align}
\label{eq:uppboundfuturemax}
\evp{g ( S_{\infty} ) \mid \cF_t}{} 
= \evp{g \lrp{\max(S_t , S_{\geq t})} \mid \cF_t}{} 
\leq \evp{g \lrp{\max \lrp{S_t , \frac{M_t}{\mathcal{U}} }} }{} := Y_t
\end{align}
The rest of the proof consists of writing the right-hand side of \eqref{eq:uppboundfuturemax} in equivalent forms. 

To prove parts $(a)$ and $(b)$, observe that 
\begin{align*}
Y_t 
&= \evp{g \lrp{\max \lrp{S_t , \frac{M_t}{\mathcal{U}} }} }{} 
= \int_0^{M_t / S_t} g \left( \frac{M_t}{s} \right) ds + \left( 1 - \frac{M_t}{S_t} \right) g (S_t) \\
&\stackrel{(\chi_1)}{=} M_t \int_{S_t}^{\infty} \frac{g (x)}{x^2} dx + \left( 1 - \frac{M_t}{S_t} \right) g (S_t) \nonumber \\
&\stackrel{(\chi_2)}{=} \frac{M_t}{S_t} \int_0^{1} g \left( \frac{S_t}{u} \right) du + \left( 1 - \frac{M_t}{S_t} \right) g (S_t) 
\end{align*}
where $(\chi_1)$ uses the change of variables $x := M_t / s$, and $(\chi_2)$ uses the change of variables $u := s S_t / M_t$. 

To prove $(c)$, start from $(\chi_1)$: 
\begin{align*}
Y_t
&= \left( 1 - \frac{M_t}{S_t} \right) g (S_t) 
+ M_t \int_{S_t}^{\infty} \frac{g (x)}{x^2} dx \\
&= \frac{1}{S_t} \left( S_t - M_t \right) g (S_t) + S_t \int_{S_t}^{\infty} \frac{g (x)}{x^2} dx 
- (S_t - M_t) \lrp{ \int_{S_t}^{\infty} \frac{g (x)}{x^2} dx } \\
&\stackrel{(\chi_3)}{=} \cG (S_t) - (S_t - M_t) \lrp{ \int_{S_t}^{\infty} \frac{g (x)}{x^2} dx - \frac{ g (S_t)}{S_t} } \\
&= \cG (S_t) - (S_t - M_t) \lrp{ \int_{S_t}^{\infty} \frac{dx}{x^2} \lrp{ g (x) - g (S_t) } }
\end{align*}
where $(\chi_3)$, like $(\chi_2)$, uses the change of variables $u := S_t / x$ to construct $\cG (S_t) = S_t \int_{S_t}^{\infty} \frac{g (x)}{x^2} dx $. To prove $(d)$, start from part $(a)$ of the result: 
\begin{align*}
Y_t &= \left( 1 - \frac{M_t}{S_t} \right) g (S_t) + M_t \int_{S_t}^{\infty} \frac{g (x)}{x^2} dx 
= g (S_t) + M_t \int_{S_t}^{\infty} \frac{dx }{x^2} \lrp{g (x) - g (S_t)}
\end{align*}
This proves that $Y_t \geq g (S_t) + M_t \cG' (S_t)$, which is $\geq 0$ if $g \geq 0$. $\cG$ is continuous because $g$ is. The concavity and monotonicity of $\cG$ are because $\cG' (s) = \int_{s}^{\infty} \frac{dx}{x^2} \lrp{ g (x) - g (s) }$ is never negative, and is monotone nondecreasing due to the monotonicity of $g$. 

This also shows that $g (S_t) = \cG (S_t) - S_t \cG' (S_t)$ (previously proved with real analysis, in \cite{CEO12}, Lemma 4.4).
\end{proof}

\begin{proof}[Proof of Lemma \ref{lem:aydifferences}]
\begin{enumerate}[(a)]
    \item
    $Y_{t} \stackrel{(a)}{=} g (S_t) + M_t \cG' (S_t) \geq g (S_t)$, 
    where $(a)$ uses Theorem \ref{thm:AYpotentialdecomp}(d). 
    \item
    When $M_t \neq S_t$ and therefore $S_{t-1} = S_t$, then $Y_{t} - Y_{t-1} = (M_t - M_{t-1})  \cG' (S_{t-1})$. 
    When $M_t = S_t$, 
    \begin{align*}
    Y_{t} - Y_{t-1} 
    &= \cG (S_t) - \cG (S_{t-1}) + (S_{t-1} - M_{t-1}) \cG' (S_{t-1}) \\
    &= \cG (S_t) - \cG (S_{t-1}) + (M_{t} - M_{t-1}) \cG' (S_{t-1}) + (S_{t-1} - S_{t}) \cG' (S_{t-1}) \\
    &= (M_{t} - M_{t-1}) \cG' (S_{t-1}) + \lrp{ \cG (S_t) - \cG (S_{t-1}) - (S_{t} - S_{t-1}) \cG' (S_{t-1}) } \\
    &= (M_{t} - M_{t-1}) \cG' (S_{t-1}) - D_{- \cG} (S_{t}, S_{t-1})
    \end{align*}
    which also shows that $Y$ is a supermartingale whenever $M$ is.
    \item
    To prove the equality, define the times at which $M_t$ sets cumulative record maxima ($M_t = S_t$) as $\tau_1 < \tau_2 < \dots \in \lrsetb{ t : M_t = S_t}$, where $Y_{\tau_{i}} = \cG (S_{\tau_{i}})$. For $v \in [\tau_{i}, \tau_{i+1})$, by definition of $Y$, $Y_{v} = \cG (S_{v}) + (M_{v} - S_{v}) \cG' (S_{v}) \leq \cG (S_{v}) = \cG (S_{\tau_{i}})$, with equality exactly at each $\tau_{i}$. Therefore, $\tau_1, \tau_2, \dots$ are also precisely the times $Y_t$ sets cumulative record maxima, and $\max_{s \leq t} Y_s = \cG (S_t)$ for all $t$. \\
    
    Now we prove the inequalities. By concavity of $\cG$ (Theorem \ref{thm:AYpotentialdecomp}), we have $Y_{t} = \cG (S_t) + (M_t - S_t) \cG' (S_{t}) \geq \cG (M_t)$. Also by monotonicity of $\cG$, $(M_t - S_t) \cG' (S_{t}) \leq 0$, so $Y_{t} \leq \cG (S_t)$.
    \item
    \label{item:optlbproc}
    $\cG$ is nondecreasing and $A_t \geq \cG (M_t)$, so that $\displaystyle \max_{s \leq t} A_s \geq \max_{s \leq t} \cG (M_s) = \cG (S_t) = \max_{s \leq t} Y_s $, using part $(b)$ for the last equality. 
\end{enumerate}
\end{proof}

\begin{proof}[Proof of Theorem \ref{thm:ayprocessproperties}]
\begin{enumerate}[a)]
\item
Using the definition of $\cG$ and Lemma \ref{lem:aydifferences}, we deduce $\max_{s \leq t} Y_s = \cG \lrp{ S_t } \leq \cG \lrp{ S_{\infty} } = \cGar^{\mu} \lrp{ \frac{1}{S_{\infty}} }$. As the $\cGar$ function is nondecreasing, using Thm. \ref{thm:supermartdecomp}\ref{item:doobineqpastfuture}, we get $\cGar^{\mu} \lrp{ \frac{1}{S_{\infty}} } \preceq \cGar^{\mu} ( \mathcal{U}) \sim \mu^{\textsc{HL}}$.
\item
The stopping event is equivalent to $ Y_{\tau^{\mu}} \leq g^{\mu} \lrp{ S_{\tau^{\mu}} } \leq g^{\mu} \lrp{ S_{\infty} } = \bar{\mu}^{-1} (1 / S_{\infty}) \preceq \bar{\mu}^{-1} (\mathcal{U})$. By definition of the tail quantile function, $\bar{\mu}^{-1} (\mathcal{U})$ has distribution $\mu$. 
\end{enumerate}
\end{proof}

To prove Theorem \ref{thm:givenmuoptisay}, we use a variational characterization of $\bar{\mu}^{\textsc{HL}}$.

\begin{proposition}[Prop. 4.10(c), \cite{CEO12}]
\label{prop:varhltrans}
\begin{align*}
\bar{\mu}^{\textsc{HL}} (y) = \min_{z>0} \frac{1}{z} \evp{(X - (y-z))^{+} }{X \sim \mu}
\end{align*}
\end{proposition}

\begin{proof}[Proof of Theorem \ref{thm:givenmuoptisay}]
We adapt an argument from \cite{brown2001robust}, via \cite{CEO12}. 
Define $Q_{\tau} := \max_{s \leq \tau} A_{s}$. Let $K < y$, and $\cG (z) = \frac{(z-y)^{+} }{y-K}$. 
Then the corresponding AY process is also nonnegative: 
$\cG (Q_{\tau}) - (Q_{\tau} - A_{\tau}) \cG' (Q_{\tau}) = \frac{ \ifn (Q_{\tau} \geq y) }{y-K} \lrp{ (Q_{\tau} - y) - (Q_{\tau} - A_{\tau}) } = \frac{ \ifn (Q_{\tau} \geq y) }{y-K} \lrp{ A_{\tau} - y } \geq 0$. Therefore, 
\begin{align*}
\ifn (Q_{\tau} \geq y)
&= \ifn (Q_{\tau} \geq y) \lrp{ \frac{ A_t - K }{y-K} + \frac{ y - A_t }{y-K} } 
\leq \ifn (Q_{\tau} \geq y) \lrp{ \frac{ A_t - K }{y-K} } 
\leq \frac{ (A_t - K)^{+} }{y-K}
\end{align*}
Taking expectations on both sides gives 
$$ \prp{ Q_{t} \geq y }{} \leq \frac{1}{y-K} \evp{(X - K)^{+} }{X \sim A_{\tau}} \stackrel{(a)}{\leq} \frac{1}{y-K} \evp{(X - K)^{+} }{X \sim \mu}$$
where $(a)$ is by the assumption $A_{\tau} \preceq \mu$ and the definition of stochastic dominance. 
This holds for any $K$. 
Minimizing over $K < y$ and using Prop. \ref{prop:varhltrans} gives $\prp{ Q_{t} \geq y }{} \leq \bar{\mu}^{\textsc{HL}} (y) = \prp{ X \geq y }{X \sim \mu^{\textsc{HL}} }$, 
yielding the result.
\end{proof}

\begin{proof}[Proof of Lemma \ref{lem:bachelier}]
It suffices to prove that if $B_t$ is defined as specified, its differences have the specified properties, which together with the initial conditions define the process uniquely almost surely.
By part $(a)$ and then parts $(b,c)$ of Lemma \ref{lem:aydifferences}, 
$ B_{t} - B_{t-1} = Y_{t} - Y_{t-1} + D_{- \cG} (S_{t}, S_{t-1}) = (M_{t} - M_{t-1}) \cG' (S_{t-1})$. 
So $B_t = Y_t + \sum_{s=1}^{t} D_{- \cG} (S_{s}, S_{s-1})$. Taking the cumulative maximum of both sides, 
$\max_{s \leq t} B_s = \max_{s \leq t} Y_s + \sum_{s=1}^{t} D_{- \cG} (S_{s}, S_{s-1}) \stackrel{(a)}{=} \cG (S_t) + B_t - Y_t = B_t + (S_t - M_t) \cG' (S_t)$, 
where $(a)$ uses Lemma \ref{lem:aydifferences} and the definition of $B$. 
\end{proof}

\begin{proof}[Proof of Theorem \ref{thm:revbachelier}]
Recall $\cG' \geq 0$ by concavity of $\cG$. The definitions in \eqref{eq:revbachelier} imply \eqref{eq:bacheq}, with initial conditions for $M_0, S_0$ specified. Given $B_t$ (w.p. 1) and $\cF_{t-1}$, this unambiguously specifies $M_t$, and hence $S_t$. The decomposition is therefore a.s. unique (as can also be shown by contradiction). 
\end{proof}

\subsection{Notes on the definitions}

\paragraph{Definition \ref{defn:fsd}. }
$(a) \implies (b)$ follows by computing the expectations on each side of $(b)$ by sampling a uniform $[0,1]$ r.v. $\mathcal{U}$ and applying $(a)$ on this variable. $(b) \implies (a)$ follows by setting $h(s) = \ifn \lrp{s \geq c}$ for any $c$. 

\paragraph{Definition \ref{defn:supq}. }
We prove the form of $\cG^{\mu}$: $\int_0^{1} g^{\mu} \left( \frac{x}{u} \right) du = \int_{0}^{1} \bar{\mu}^{-1} (u/x) du = x \int_{0}^{1/x} \bar{\mu}^{-1} (\lambda) d \lambda = \cGar^{\mu} (1/x)$.

\section{Prototypical example: $z$-test and sub-Gaussian statistics}
\label{sec:ztest}

For sub-Gaussian statistics, characterization of their null distributions often ultimately relies on the Central Limit Theorem (CLT). 
Therefore, we use the $z$-test as a prototypical example to introduce the  concentration behaviors.

\subsection{A $p$-value for a fixed time}
The $z$-test's statistic, appropriately normalized, is a sum of standard normal random variables $M_t = \sum_{i=1}^{t} Z_i$, and its moment-generating function (m.g.f.) of any $\lambda \in \RR$ is $\evp{ e^{\lambda M_t} }{} = e^{\frac{\lambda^2}{2} t}$. 
So the variable $Y_t^{\lambda} := e^{\lambda M_t - \frac{\lambda^2}{2} t }$ has mean $1$, and Markov's inequality tells us that $A_t^{\lambda} := \frac{1}{Y_t^{\lambda}}$ meets the above definition of a $p$-value, i.e. $\prp{A_t^{\lambda} \leq s}{} \leq s \;\;\forall s$. 

All this holds for any $\lambda$, so the best $p$-value at a fixed time is $A_t = \min_{\lambda} A_t^{\lambda} = e^{- M_t^2 / 2t}$, recovering the well-known Gaussian tail. 
Peeking can be disastrous in this canonical scenario, leading to a profusion of false positives (indeed, classical results \citep{AMR69} prove that a peeker willing to wait for the $z$-test as long as necessary can report any desired $p < 1$ w.p. 1).

\subsection{Inoculation against peeking by mixing distributions}
\label{sec:ztestinocarg}

To devise such a peeking-robust $H$, recall the distribution of $M$ as specified by its m.g.f. at all times: $\evp{ e^{\lambda M_t - \frac{\lambda^2}{2} t } }{} = 1 \;\;\forall \lambda$. 
So for any distribution $\Gamma$, we have $1 = \evp{ W_{t}^{\Gamma} }{}$ for the mixed process $W_{t}^{\Gamma} := \evp{ e^{\lambda M_t - \frac{\lambda^2}{2} t } }{\lambda \sim \Gamma}$, and the process $W_{t}^{\Gamma}$ is a nonnegative martingale.
Its expectation is controlled at any stopping time $\tau$ by the optional stopping theorem (Theorem \ref{thm:optstoppingsupermart}), so $\evp{ W_{\tau}^{\Gamma} }{} = \evp{ W_{0}^{\Gamma} }{} = 1$. 
Therefore, defining $H_{\tau}^{\Gamma} := \frac{1}{W_{\tau}^{\Gamma}}$ and using Markov's inequality on $W_{\tau}^{\Gamma}$, we have $\prp{ H_{\tau}^{\Gamma} \leq s}{} = \prp{ W_{\tau}^{\Gamma} \geq \frac{1}{s}}{} \leq s$, so $H_{\tau}^{\Gamma}$ behaves like a $p$-value despite the arbitrariness of the stopping time $\tau$. 
This is true regardless of the distribution $\Gamma$, which controls how the reported $H$ varies over each \emph{sample path} \citep{B14, HRMS18}.


Such pathwise variation is unavoidably $\Omega (\sqrt{t \log \log t})$, the content of a fundamental theorem of probability -- the \emph{law of the iterated logarithm} (LIL). 
Proofs of the asymptotic \citep{RS70} and finite-time LIL \citep{B14} have used its relationship with mixed processes like $X_{\tau}^{\Gamma}$, and that line of work has explored how best to choose $\Gamma$ \citep{HRMS18}. 

\subsection{Robust $p$-values for sub-Gaussian statistics}

Despite their generality, (super)martingales whose increments are sub-Gaussian follow concentration behavior like $M_t$, the Gaussian random walk of the $z$-statistic we have discussed.
The recipe for $H$-values is much the same for these generalizations, 
where $W_t = \evp{ e^{\lambda M_t - \frac{\lambda^2}{2} V_t } }{\lambda \sim \Gamma}$ is a (super)martingale for different values of $\lambda$, with $V_t$ being the martingale's cumulative variance process. 
So $H_{t}^{\Gamma} 
:= \lrp{\evp{ e^{\lambda M_t - \frac{\lambda^2}{2} V_t } }{\lambda \sim \Gamma}}^{-1}$ satisfies $\prp{H_{\tau}^{\Gamma} \leq s}{} \leq s$ for any stopping time $\tau$ -- by the argument of \Cref{sec:ztestinocarg}. 
This makes it a robust $p$-value.
\footnote{The guarantees on $H$ hold even at the time of the \emph{ultimate} minimum of $H$ (see \Cref{sec:multreprmart}). 
This is not a stopping time, as it depends on future events, and was originally termed an ``honest time" \citep{Nikeghbali07, Nikeghbali13}. 
Following this, a robust $p$-value is also ``honest."}

\end{document}